\documentclass[pdflatex,sn-mathphys-num]{sn-jnl}% Math and Physical Sciences Numbered Reference Style
\usepackage{graphicx}%
\usepackage{multirow}%
\usepackage{amsmath,amssymb,amsfonts}%
\usepackage{amsthm}%
\usepackage{mathrsfs}%
\usepackage[title]{appendix}%
\usepackage{xcolor}%
\usepackage{textcomp}%
\usepackage{manyfoot}%
\usepackage{booktabs}%
\usepackage{algorithm}%
\usepackage{algorithmicx}%
\usepackage{algpseudocode}%
\usepackage{listings}%
\theoremstyle{thmstyleone}%
\theoremstyle{thmstyletwo}%

\theoremstyle{thmstylethree}%

\theoremstyle{definition}

\newtheorem{Exm}[subsection]{Example}
\newtheorem{Rem}[subsection]{Remark}

\newtheorem{Thm}[subsection]{Theorem}
\newtheorem{Lem}[subsection]{Lemma}

\DeclareMathOperator{\Z}{\mathbb{Z}}

\DeclareMathOperator{\rk}{\text{\rm rank}}
\DeclareMathOperator{\ch}{\text{\rm ch}}
\DeclareMathOperator{\Q}{\mathbb{Q}}
\DeclareMathOperator{\tens}{\otimes}
\DeclareMathOperator{\Ker}{\text{\rm ker}}

\usepackage{tikz-cd}

\begin{document}

\title[Article Title]{$C$-triviality of manifolds of low dimensions}


\author[1]{\fnm{Shubham} \sur{Sharma}}\email{sharm340@msu.edu}

\author[2]{\fnm{Animesh} \sur{Renanse}}\email{arenanse@ucsc.edu}


\affil[1]{\orgdiv{Department of Mathematics}, \orgname{Michigan State University}, \orgaddress{\street{220 Trowbridge Rd}, \city{East Lansing}, \postcode{48824}, \state{Michigan}, \country{USA}}}

\affil[2]{\orgdiv{Department of Mathematics}, \orgname{University of California, Santa Cruz}, \orgaddress{\street{1156 High St}, \city{Santa Cruz}, \postcode{95064}, \state{California}, \country{USA}}}

%\affil[3]{\orgdiv{Department}, \orgname{Organization}, \orgaddress{\street{Street}, \city{City}, \postcode{610101}, \state{State}, \country{Country}}}


\abstract{A space $X$ is said to be $C$-trivial if the total Chern class $c(\alpha)$ equals $1$ for every complex vector bundle $\alpha$ over $X$. In this note we give a complete homological classification of $C$-trivial closed connected smooth manifolds of dimension $\le 7$. Our main tool is the Atiyah-Hirzebruch spectral sequence and orders of its differentials.}




\maketitle

\section{Introduction}
%%%%%%%%%%%%%%%%%%%%%%%%%%%%%%%%%%%%%%%%%%%%%%%%%%%%%%%%%%%%%%%%%

A space $X$ is said to be $C$-trivial if the total Chern class $c(\alpha)$ equals $1$ for every complex vector bundle $\alpha$ over $X$. 
There are analogous definitions of a $W$-trivial (respectively, $P$-trivial) space, to describe spaces where the total Stiefel-Whitney class $w(\alpha)$ (respectively the total Pontrjagin class $p(\alpha)$) equals $1$ for every real vector bundle $\alpha$ over $X$. 

Given a space $X$, it is an interesting question to understand whether or not $X$ is $W$-trivial, $P$-trivial or $C$-trivial. 
In recent times several authors have investigated this question.  We refer the reader to \cite{bhatta}, \cite{tanaka1}, \cite{tanaka}, 
\cite{ajaydold}, \cite{chern}, \cite{podder} and the references therein. 

One of the first theorems in this direction is the theorem of Atiyah-Hirzebruch. 

\begin{Thm} \cite[Theorem 2, page 223]{atiyah} For a finite $CW$-complex $X$, the $9$-fold suspension $\Sigma^9X$ is $W$-trivial. 
	\qed
\end{Thm}

The above theorem implies that for a finite $CW$-complex $X$, the suspension $\Sigma^kX$ is $W$-trivial whenever $k\geq 9$. 
Tanaka, in a series of papers (see \cite{tanaka}, \cite{tanaka1}, \cite{tanaka2}, \cite{tanaka3}), investigated the $W$-triviality of iterated suspensions of projective spaces 
(over $\mathbb R, \mathbb C$ and $\mathbb H$).  In \cite{anistunted}, the authors have determined conditions under which the iterated suspension 
$\Sigma^k(\mathbb RP^m/\mathbb RP^n)$ of the stunted  real projective spaces is $W$-trivial. The $W$-triviality of the iterated suspensions of the Dold manifolds has been determined  in the paper \cite{ajaydold}. 

Not much discussion is available in the literature about $C$-trivial and $P$-trivial spaces. We refer the reader to \cite{chern} for a discussion on $C$-triviality and very recently to \cite{podder} for a discussion of $P$-triviality. In \cite{chern} the authors completely determine which iterated suspensions 
$\Sigma^k(\mathbb FP^m/\mathbb FP^n)$ of the stunted projective spaces, where $\mathbb F= \mathbb R,\mathbb C,\mathbb H$, are $C$-trivial. 
In \cite{podder}, the authors obtain a complete description of when $\Sigma^k(\mathbb FP^m/\mathbb FP^n)$ is $P$-trivial when $\mathbb F=\mathbb R,\mathbb C$. 

A recent paper \cite{bhatta} discusses $W$-triviality of low dimensional manifolds. The authors obtain almost a complete description of closed smooth manifolds that are $W$-trivial in each dimension $n\leq 7$. For example, the authors prove in \cite[Theorem 1.4]{bhatta} that a closed orientable smooth manifold  $X$ of dimension $n=3,5$ is $W$-trivial if and only if $X$ is a $\mathbb Z_2$-homology sphere. Recall that for a commutative ring $R$ with $1$, a $R$-homology $n$-sphere is a closed connected smooth $n$-manifold $X$ such that 
$H_i(X;R)\cong H_i(S^n;R)$ for all $i\geq 0$.

In this note we try to understand which closed smooth manifolds are $C$-trivial. The Bott integrality theorem places an obstruction to the 
$C$-triviality of a closed orientable smooth $n$-manifold if $n$ is even. Indeed, by the Bott integrality theorem, if $\alpha$ is a complex vector bundle over $S^{2n}$, then $c_n(\alpha)$ is divisible by $(n-1)!$ and conversely. Thus no even dimensional sphere is $C$-trivial. It is well known that if $X$ is a closed orientable smooth $n$-manifold then there exists a degree one map $f:X\longrightarrow S^n$. This in conjunction with the Bott integrality theorem implies that no even dimensional closed orientable smooth manifold is $C$-trivial. In odd dimensions, integral homology spheres provide examples of manifolds that are $C$-trivial. In low dimensions we can say a lot more, often leading to a complete homological description of when a closed smooth manifold is $C$-trivial. 

In this note we try to derive, whenever possible, a complete description of when a closed smooth $n$-manifold is $C$-trivial, $n\leq 7$. 
Before stating the main results of this note we make a few remarks. For obvious reasons, every closed smooth $1$-manifold is $C$-trivial. 
Also, a necessary condition for a space $X$ to be $C$-trivial is that we must have $H^2(X;\mathbb Z)=0$ (see Lemma \ref{PreliminaryResults-1} below). This immediately implies that no compact surface is $C$-trivial. 

We now state the main results. In what follows, all manifolds are assumed to be closed connected and smooth. Our results are of two types: the general, and the ones specific to manifolds of dimension at most $7$. We begin with the general results. 

\begin{Thm}\label{GR_Even_Cohomology_Finite} Let $X$ be a $C$-trivial $n$-manifold. Then $H^{2q}(X;\mathbb Z)$ is a finite abelian 
	group for all $2 \le 2q < n$. \qed
	%$i\geq 1$
	
\end{Thm}

\begin{Thm}\label{GR_Odd_Dim_Manifolds}
	Let $X$ be a $C$-trivial $n$-manifold with $n$ odd. 
	\begin{enumerate}
		\item If $X$ is orientable, then $H^i(X;\mathbb Z)$ is finite for all $i$, $0<i<n$,
		\item If $X$ is non-orientable, then $H^1(X;\mathbb Z)\cong \mathbb Z$ and $H^i(X;\mathbb Z)$ is finite for all $i$, $2\leq i\leq n$. 
	\end{enumerate}
\end{Thm}
%Theorem \ref{GR_Odd_Dim_Manifolds} at once gives a complete description of $C$-trivial $3$-manifolds. We also show that a $4$-manifold can never be $C$-trival. For dimensions $5$ to $7$, we have the following results.

We now state results that identify $C$-trivial manifolds of dimension at most 7. As discussed above, there are no $C$-trivial manifolds in dimension $1$ and $2$. In dimension $3$, we show that an orientable $3$-manifold is $C$-trivial if and only if it is an integral homology $3$-sphere (see Theorem \ref{3_Manifold} below) and in the non-orientable case we give homological restrictions on a $3$-manifold to be $C$-trivial (see Theorem \ref{3_Manifold} below). Next we show that no $4$-manifold is $C$-trivial (see Theorem \ref{4_Manifold} below). For dimensions $5$, $6$ and $7$, the results are as follows.

%\begin{Thm} \label{3_Manifold}
%Let $X$ be a $3$-manifold. 
%\begin{enumerate}
%\item If $X$ is orientable, then $X$ is $C$-trivial if and only if $X$ is an integral homology $3$-sphere.
%\item If $X$ is non-orientable, then $X$ is $C$-trivial if and only if the integral homology groups of $X$ are of the form 
%$$H_i(X;\mathbb Z)=\left\{ \begin{array}{cl}
	%\mathbb Z & i=0\\
	%\mathbb Z & i=1\\
	%\mathbb Z_2 & i=2\\
	%0 & i=3
	%\end{array}
	%\right. $$ \qed
	%\end{enumerate}
	%\end{Thm}
	
	%There  are examples of $3$-manifolds whose integral homology has the form described in the theorem above (see \cite[Table 3, page 573]{burton}).
	
	%\begin{Thm}\label{4_Manifold}
	%Let $X$ be a $4$-manifold. Then $X$ is not $C$-trivial. \qed
	%\end{Thm} 
	
	\begin{Thm} \label{5-Manifold-Orientable} Let $X$ be an orientable $5$-manifold. Then $X$ is $C$-trivial if and only if the integral homology groups of $X$ are of the form
		$$
		H_i(X;\mathbb Z) = \left\{ \begin{array}{cl}
			\mathbb Z & i=0\\
			0 & i=1\\
			F & i=2\\
			0 & i=3\\
			0 & i=4\\
			\mathbb Z & i=5. 
		\end{array}\right. $$ 
		where $F$ is a finite abelian group. 
	\end{Thm}
	In Example \ref{Orientable_5_7_Manifold_Example}, we provide oriented manifolds of dimensions $5$ which are $C$-trivial via the above classification.
	In the non-orientable case we have the following statement.
	
	\begin{Thm} \label{5-Manifold-Non-Orientable} Let $X$ be a non-orientable $5$-manifold. Then $X$ is $C$-trivial if and only if the integral homology groups of $X$ are of the form
		$$
		H_i(X;\mathbb Z)=\left\{ \begin{array}{cl}
			\mathbb Z & i=0\\
			\mathbb Z& i=1\\
			F & i=2\\
			0 & i=3\\
			\mathbb Z_2 & i=4\\
			0 & i=5. 
		\end{array}\right.$$
		where $F$ is a finite abelian group.
	\end{Thm}
	
	As noted earlier, in dimension $6$, no closed connected orientable manifold is $C$-trivial. We prove the following for the non-orientable case.
	\begin{Thm} 
		\label{6-Manifold-NonOrientable} Let $X$ be a non-orientable $6$-manifold. Then $X$ is $C$-trivial if and only if the integral homology groups of $X$ are of the form
		$$
		H_i(X;\mathbb Z) = \left\{ \begin{array}{cl}
			\mathbb Z  & \text{if } i = 0\\
			\mathbb Z^{e_2} & \text{if } i = 1\\
			F & \text{if } i = 2\\
			\mathbb Z^{e_3} & \text{if } i = 3\\
			F' & \text{if } i = 4\\
			\mathbb Z^{e_1} \oplus \mathbb Z_2 & \text{if } i = 5\\  
			0 & \text{if } i = 6\\  
		\end{array}\right.$$
		where $e_2 \neq 0$, $F,F'$ are finite abelian groups such that $\mathrm{Ext}(F,\mathbb Z_2)\cong \mathrm{Ext}(F',\mathbb Z_2)$.
	\end{Thm}
	For orientable $7$-manifolds we prove the following. 
	
	\begin{Thm} \label{7_manifold_orientable} Let $X$ be an orientable $7$-manifold. Then $X$ is $C$-trivial if and only if the integral homology groups of $X$ are of the form
		$$
		H_i(X;\mathbb Z) = \left\{ \begin{array}{cl}
			\mathbb Z & i=0\\
			0 & i=1\\
			F & i=2\\
			0 & i=3\\
			F & i=4\\
			0 & i=5\\
			0 & i=6\\
			\mathbb Z & i=7\\ 
		\end{array}\right.$$
		where $F$ is a finite abelian group.
	\end{Thm}
	In Example \ref{Orientable_5_7_Manifold_Example}, we provide oriented manifolds of dimensions $7$ which are $C$-trivial via the above classification.
	
	In the $7$-dimensional non-orientable case we prove the following. 
	
	\begin{Thm} \label{7_manifold_non_orientable} Let $X$ be a non-orientable $7$-manifold. Then $X$ is $C$-trivial if and only if $X$ has the homology
		$$
		H_i(X;\Z) = \begin{cases}
			\Z  & \text{if } i = 0\\
			\Z & \text{if } i = 1\\
			F & \text{if } i = 2\\
			\Z_2^r & \text{if } i = 3\\
			F' & \text{if } i = 4\\
			0 & \text{if } i = 5\\  
			\Z_2 & \text{if } i = 6\\
			0 & \text{if } i = 7
		\end{cases}
		$$
		where $F,F'$ are finite abelian groups, $\mathrm{Ext}(F,\mathbb Z_2)\cong \mathrm{Ext}(F',\mathbb Z_2)$ and either one of the following holds
		\begin{enumerate}
			\item [$(i)$]{$r=0$, or}
			\item [$(ii)$]{$r=1$ with $\mathrm{Sq^2\circ\rho_2}:H^4(X;\Z) \rightarrow H^6(X;\Z_2)$ being an injective map.}
		\end{enumerate}
	\end{Thm}
	
	{\em Conventions.} We follow the same conventions as in \cite{bhatta} and record them here for convenience. 
	Throughout $F,F',\ldots$ will denote finite abelian groups. Given a finite abelian group $F$ (respectively $F',\ldots$), 
	the integers $s$ (respectively, $s'\ldots$) will denote the number of primes $p_i$ that are equal to $2$ in a decomposition 
	$$F\cong \oplus \mathbb Z_{p_i}^{k_i}$$
	of $F$. 
	
	%%%%%%%%%%%%%%%%%%%%%%%%%%%%%%%%%%%%%%%%%%%%%%%%%%%%%%%%%%%%%%%%%
	\section{Proof of general results}
	%%%%%%%%%%%%%%%%%%%%%%%%%%%%%%%%%%%%%%%%%%%%%%%%%%%%%%%%%%%%%%%%%

	Before proving the main theorems, we introduce some notations and prove some preliminary results that we shall need. Throughout, $\rho_k$ will denote the homomorphism 
	$$\rho_k:H^k(X;\mathbb Z)\longrightarrow H^k(X;\mathbb Z_k)$$
	induced by the quotient map $\Z \to\Z_k$. We begin by some elementary, but important observations.
	\begin{Lem} \label{PreliminaryResults-1} Let $X, Y$ be two paracompact spaces. 
		\begin{enumerate}
			\item If $X$ is $C$-trivial, then $H^2(X;\Z) = 0$. Hence, $H_2(X;\Z)$ is a finite group and $H_1(X;\Z) \cong \Z^e$ for some $e \geq 0$. 
			\item If $\widetilde{K}(X)=0$, then $X$ is $C$-trivial.
			\item If $\widetilde{K}(X)=0=\widetilde{K}(Y)$ and $X\wedge Y$ is $C$-trivial, then $X\times Y$ is $C$-trivial. 
			\item If $n$ is odd and $X$ has trivial complex $K$-theory, that is $\widetilde{K}(\Sigma X)=0= \widetilde{K}(X)$, then $S^n\times X$ is $C$-trivial.
		\end{enumerate}
	\end{Lem} 
	\begin{proof}
		The proofs of (1) and (2) are straightforward. To prove (3), we observe that in the exact sequence
		\begin{align*}
			\widetilde{K}(X\wedge Y) \overset{p^*}{\to} \widetilde{K}(X\times Y) \to \widetilde{K}(X\vee Y).
		\end{align*}
		$p^*$ is surjective as $\widetilde{K}(X\vee Y) = 0$. Consequently, every bundle on $X\times Y$ is stably isomorphic to the pullback of a bundle from $X\wedge Y$. The result then follows by naturality of Chern classes.\newline
		We now prove (4). As $\widetilde{K}(S^n) = 0$ for $n$ odd, by (3) it is sufficient to show that $\widetilde{K}(\Sigma^nX) =0$ for odd $n$. The result then follows by Bott periodicity, since we have $\widetilde{K}(\Sigma^nX) \simeq \widetilde{K}(\Sigma^{n-2}X) \simeq \dots  \simeq \widetilde{K}(\Sigma X)$. \end{proof}
	For a space $X$, let $d_i$ denote the $i^{\text{\rm th}}$-coboundary homomorphism of the Atiyah-Hirzebruch spectral sequence of complex $K$-theory of $X$. For easy reference, we state a result regarding these differentials that we will often use in our proofs. The references for this result are \cite[Theorem 3]{Buhstaber}, \cite[Remark 1.4]{Banica_2} and \cite[Theorem 1, pp 172]{griffiths}.
	\begin{Thm}\label{BoundaryMap_AHSS_Kernel}\label{ChernClassExistence_Griffiths} Let $X$ be a finite polyhedron and let $d_{2k+1}$ denote the (possible) non-trivial coboundaries in the Atiyah-Hirzebruch spectral sequence for $K(X)$, associated with the simplicial decomposition of $X$. Fix $q \ge 1$. If $\alpha \in H^{2q}(X;\Z)$ lies in the kernel of 
		\begin{align*}
			d_{2k+1} : H^{2q}(X;\Z) \to H^{2q+2k+1}(X;\Z)
		\end{align*}
		for all $k \ge 1$, then there exists a vector bundle $\xi$ over $X$ such that $c_q(\xi) = (q-1)!\alpha$.
		\newline
		On the other hand, if there exists a vector bundle $\xi \in K(X)$  which is trivial on $X^{2q-1}$, the $2q-1$ skeleton of $X$, and is such that 
		$$
		\ch(\xi) = \alpha + \text{higher order terms}.
		$$
		then there exists a cohomology class $\alpha$ in $H^{2q}(X;\Z)$ such that $d_{2k+1}(\alpha) = 0$ for all $k$. The map $\ch$ denotes the Chern character map.
	\end{Thm}
	
	We now provide the proofs of the two general results given in section 1.
	\begin{proof}[Proof of Theorem \ref{GR_Even_Cohomology_Finite}]
		Let $X$ be a closed $n$-manifold which is $C$-trivial. We prove the result by contradiction. If possible, let $\alpha \in H^{2q}(X;\Z)$ be an element of infinite order in $H^{2q}(X;\Z)$. We shall find a non-zero element of $H^{2q}(X;\Z)$ which is the $q^{th}$-Chern class of a complex vector bundle $\xi$ over $X$.\newline
		Let $2m+1$ be the largest odd integer such that $2q + 2m + 1 \le n$. For $1 \le k \le m$, $d_{2k+1}:H^{2q}(X;\Z) \rightarrow H^{2q+2k+1}(X;\Z)$ is an odd differential of the Atiyah-Hirzebruch spectral sequence.  It is well known that the image of the coboundary homomorphisms of the Atiyah-Hirzebruch spectral sequence is torsion-valued \cite{Buhstaber}. So, let $n_k$ be the smallest positive integer such that $n_kd_{2k+1}(x) = 0$ for all $x \in H^{2q}(X;\Z)$. In particular, $n_kd_{2k+1}(\alpha) = d_{2k+1}(n_k \alpha) = 0$ for all integers $k$, $1 \le k \le m$.\\ 
		Let $l = \prod_{k=1}^m{n_k}$. Then it is clear that $d_{2k+1}(l\alpha) = 0$ for all $1 \le k \le m$. For $k > m$, $2q + 2k+1 > n = \dim(X)$ and hence $d_{2k+1}(l \alpha)$ is zero in any case. So, the element $\mu = l \alpha$ is in $\Ker(d_{2k+1})$ for all positive integers $k$. It follows from Theorem \ref{BoundaryMap_AHSS_Kernel} that there exists a complex vector bundle $\xi$ such that $c_q(\xi) = (q-1)! l \alpha$ and since $\alpha$ is an element of infinite order, $c_q(\xi) \neq 0$. This contradicts the $C$-triviality of $X$ and completes the proof.
	\end{proof}
	\begin{Rem}
		A more direct proof of Theorem \ref{GR_Even_Cohomology_Finite} can be achieved by the Chern character isomorphism. There is an isomorphism
		\begin{align*}
			\ch : K(X)\tens \Q \to H^{\text{\rm ev}}(X;\Q)
		\end{align*}
		where $\ch_k : K(X)\tens \Q\to H^{2k}(X;\Q)$ is obtained by a homogeneous polynomial among Chern classes for $k\ge 1$. By $C$-triviality, $\ch_k = 0$ for $2 \le 2k < n$. As $\ch$ is an isomorphism, it follows that $H^{2k}(X;\Q) = 0$ for all $2 \le 2k < n$, as required.
	\end{Rem}
	\begin{proof}[Proof of Theorem \ref{GR_Odd_Dim_Manifolds}] We only prove (2), the proof of (1) follows a similar analysis. By Lemma \ref{PreliminaryResults-1}(1), $H_1(X;\Z) = \Z^e$ for some $e \ge 0$ and hence $H^1(X;\Z_2) = \Z_2^e$. By Theorem \ref{GR_Even_Cohomology_Finite}, $H_{n-1}(X;\Z) = F$ for some finite group $F$ and since $X$ is non-orientable, $H_{n-1}(X;\Z) = \Z_2$. If $H_{n-2}(X;\Z) = \Z^{e_1} \oplus F$, for some $e_1 \ge 0$ and some finite abelian group $F$, then $$H^{n-1}(X;\Z_2) = \Z_2 \oplus \Z_2^{e_1} \oplus \Z_2^{s}.$$
		Using Poincar\'e duality for $\Z_2$ coefficients, we get $H^1(X;\Z_2) = H^{n-1}(X;\Z_2)$, which implies 
		$$e = 1+e_1+s \ge 1.$$
		Hence,
		\begin{equation}
			\chi(X) = 1 - \rk(H^1(X;\Z)) + \sum_{\substack{i=1  \\ i \text{ even}}}^{n} (-1)^i \rk(H^{i}(X;\Z)) + \sum_{\substack{i=2 \\ i \text{ odd}}}^{n} (-1)^i \rk(H^{i}(X;\Z)).
		\end{equation}
		By Theorem \ref{GR_Even_Cohomology_Finite} the third term is $0$. So the expression reduces to
		\begin{equation}
			0 = \chi(X) = 1 - e  - \sum_{\substack{i=2 \\ i \text{ odd}}}^{n} \rk(H^{i}(X;\Z)).
		\end{equation}
		Hence, $e \le 1$, which combined with the earlier observation that $e \ge 1$, gives us $e = 1$. It also follows that $\rk(H^i(X;\Z)) = 0$ for all odd $i > 2$. This completes the proof of (2).
	\end{proof}
	
	\begin{Rem}\label{Remark_Odd_Orientable_Second_Last_Hom} Let $X$ be an orientable, $C$-trivial manifold of odd dimension, say $2k+1$. Since $H_{2k}(X;\Z)$ must be both torsion free and finite, it must be true that $H_{2k}(X;\Z) = 0$. Using $C$-triviality $H^2(X;\Z) = 0$ and hence, $H_{2k-1}(X;\Z) = 0$ by Poincar\'e duality. Hence, $H^{2k}(X;\Z) = 0$ and applying Poincar\'e duality we get that $H_1(X;\Z) = 0$. An interesting consequence of this is that an orientable $C$-trivial odd-dimensional manifold must have a perfect fundamental group.
	\end{Rem}

	\section{Calculations for low dimensional manifolds}
	In this section we prove Theorems \ref{5-Manifold-Orientable}-\ref{7_manifold_non_orientable}. We begin by classifying $C$-trivial manifolds in dimensions $3$ and $4$ as mentioned in the introduction.
	\begin{Thm} \label{3_Manifold}
		Let $X$ be a $3$-manifold. 
		\begin{enumerate}
			\item If $X$ is orientable, then $X$ is $C$-trivial if and only if $X$ is an integral homology $3$-sphere.
			\item If $X$ is non-orientable, then $X$ is $C$-trivial if and only if the integral homology groups of $X$ are of the form 
			$$H_i(X;\mathbb Z)=\left\{ \begin{array}{cl}
				\mathbb Z & i=0\\
				\mathbb Z & i=1\\
				\mathbb Z_2 & i=2\\
				0 & i=3.
			\end{array}
			\right. $$
		\end{enumerate}
	\end{Thm}
	\begin{proof} We begin by proving (1). Using Lemma \ref{PreliminaryResults-1}, we know that $H_1(X;\Z) \simeq \Z^e$ for some $e\ge 0$ and $H_2(X; \Z) = F$ for some finite group $F$. By orientability and connectedness of $X$, we further conclude that $H_0(X; \Z) = \Z$ and $H_3(X; \Z) = \Z$. By Poincar\'e duality, we get that $H_1(X;\Z) \simeq H^2(X; \Z) = 0$. It follows that $e=0$. Since $X$ is orientable, we further get that $F = 0$ and thus $X$ is an integral homology $3$-sphere, as required. The converse is immediate. This proves (1). 
		\newline
		We now prove (2). Assume that $X$ is a $C$-trivial non-orientable closed $3$-manifold. Observe, $H_3(X;\Z) = 0$ and $H_2(X; \Z)$ has torsion subgroup as $\Z_2$. By Lemma \ref{PreliminaryResults-1}, it follows that $H_2(X; \Z) \simeq \Z_2$. Furthermore, by Lemma \ref{PreliminaryResults-1}, we have that $H_1(X; \Z) \simeq \Z^e$ for some $e \ge 0$.
		Since $\dim(X)$ is odd, $(1-e) = \chi(X) = 0$, which gives us the required homology.
		The converse follows from the fact that if the homology is as given then $H^2(X;\Z) = 0$, completing the proof.
	\end{proof}
	\begin{Rem}
		There  are examples of non-orientable $3$-manifolds whose integral homology groups are of the form described in the Theorem \ref{3_Manifold} (see \cite[Table 3, page 573]{burton}).    
	\end{Rem}
	Before proving theorem \ref{4_Manifold} and theorems \ref{5-Manifold-Orientable}-\ref{7_manifold_non_orientable} which give the homological classification of $C$-trivial manifolds of dimensions 4,5,6 and 7, we prove a preliminary result.
	\begin{Lem}\label{PreliminaryResults-2} Let $X$ be a $CW$-complex and assume that 
		$H^7(X;\mathbb Z)$ has no $2$-torsion.  
		\begin{enumerate}
			\item  If $\dim(X)\leq 7$ and $X$ is $C$-trivial then the composition 
			\begin{equation}\label{equation1}
				H^4(X;\Z) \xrightarrow{{\rho_2}} H^4(X;\Z_2) \xrightarrow{{\mathrm{Sq}^2}} H^6(X;\Z_2)
			\end{equation}
			must be injective.
			\item If $\dim(X) \leq 6$ and $X$ is $C$-trivial, then $H^2(X;\Z) = 0$ and $H^4(X;\Z) = 0$. 
		\end{enumerate}
	\end{Lem}
	\begin{proof}
		We first prove (1). We note that, by \cite[Corollary\,2.2]{Banica_1}, a triple $(c_1,c_2,c_3)$ of cohomology classes 
		$$(c_1,c_2,c_3)\in H^2(X;\mathbb Z)\times H^4(X;\mathbb Z) \times H^6(X;\mathbb Z)$$
		are the Chern classes of a rank $3$ complex vector bundle $\alpha$ over $X$ if and only if 
		\begin{equation}\label{equation2}
			c_3\equiv c_1c_2+ Sq^2 c_2
		\end{equation}
		in $H^6(X;\mathbb Z_2)$. If the composition in (\ref{equation1}) has a non-zero element in its kernel, say $c_2$, then the 
		triple $(0,c_2,0)$ satisfies (\ref{equation2}) and hence there is a rank $3$ complex vector bundle $\alpha$ over $X$ with $c_2(\alpha)=c_2\neq 0$ and hence $X$ is not $C$-trivial. This proves (1). 
		\newline
		We next prove (2). Assume that $\dim(X)\leq 6$ and that $X$ is $C$-trivial. We first observe that there is an exact sequence 
		$$\dots \longrightarrow H^5(X;\mathbb Z_2)\longrightarrow H^6(X;\mathbb Z)\longrightarrow H^6(X ; \mathbb Z) \stackrel{\rho_2}\longrightarrow H^6(X;\mathbb Z_2)\longrightarrow H^7(X;\mathbb Z)\longrightarrow 0$$
		and hence the homomorphism $\rho_2$ is an surjective. We now consider two cases. In the case that 
		$\dim(X)<6$, the composition $Sq^2\circ \rho_2$ of (\ref{equation1}) is the zero homomorphism and it will have a non-trivial kernel if $H^4(X;\mathbb Z)\neq 0$. Then by (1), $X$ cannot be $C$-trivial which is a contradiction. If $\dim(X)=6$, we assume 
		$H^4(X;\mathbb Z)\neq 0$ and derive a contradiction. Let $c_2\in H^4(X;\mathbb Z)$ be a non-zero element. As the morphism
		$$\rho_2:H^6(X;\mathbb Z)\longrightarrow H^6(X;\mathbb Z_2)$$
		is now surjective, we find a $c_3\in H^6(X;\mathbb Z)$ with 
		$$Sq^2\circ \rho_2(c_2)=\rho_2(c_3).$$
		The triple $(0,c_2,c_3)$ now satisfies the equation (\ref{equation2}). Hence there is a rank $3$ complex vector bundle $\alpha$ over $X$ with $c_2(\alpha)=c_2$, $c_3(\alpha)=c_3$. This contradiction proves (2) 
		and completes the proof of the lemma. 
	\end{proof}
	
	\begin{Rem}
		It follows that if $X$ is a $C$-trivial $CW$-complex of dimension at most $6$, then $H_i(X;\mathbb Z)$ is a finite abelian group for $i=2,4$ and is a torsion free abelian group for $i=1,3$.
	\end{Rem}
	
	\begin{Thm}\label{4_Manifold}
		Let $X$ be a $4$-manifold. Then $X$ is not $C$-trivial.
	\end{Thm}
	\begin{proof} As remarked in the initial discussion, orientable manifolds of even dimension $n$ always admits a complex vector bundle $\alpha$ of rank $n/2$ with $c_{n/2}(\alpha) \neq 0$. So we assume $X$ is non-orientable. By Lemma \ref{PreliminaryResults-2}, we have $H^4(X;\Z) = 0$. This is a contradiction as we must have $H^4(X;\Z) = \Z_2$.
	\end{proof}
	\begin{proof}[Proof of Theorem \ref{5-Manifold-Orientable}]  Let $X$ be a $C$-trivial orientable $5$-manifold. Observe the following
		\begin{itemize}
			\item[(i)] By Theorem \ref{GR_Odd_Dim_Manifolds}, $H_i(X;\Z)$ is finite for all $0<i<5$.
			\item[(ii)] Since $H_4(X;\Z)$ is finite and torsion free, we have that $H_4(X;\Z) = 0$ and hence $H^1(X;\Z) = 0$.
			\item[(ii)] As $X$ is $C$-trivial, $H_1(X;\Z)$ is torsion free and by Theorem \ref{GR_Odd_Dim_Manifolds} it is finite and hence $H_1(X;\Z) = 0$.
			\item[(iii)] By Lemma \ref{PreliminaryResults-2}, $H_3(X;\Z)$ is free abelian and we know it must also be finite. Hence $H_3(X;\Z) = 0$. 
		\end{itemize}
		Using the above observations the homology groups of $X$ are as given in the statement of Theorem \ref{5-Manifold-Orientable}. Conversely, let $X$ have the homology given in the statement of Theorem \ref{5-Manifold-Orientable}. Then it is clear that $X$ is $C$-trivial as the cohomology groups of even degree are zero. This completes the proof of the theorem.
	\end{proof}
	The above result shows that if $X$ is a closed orientable 5-dimensional manifold which is $C$-trivial then $\pi_1(X)$ is perfect and $X$ is a rational homology sphere.
	\begin{proof}[Proof of Theorem \ref{5-Manifold-Non-Orientable}] Let $X$ be a $C$-trivial non-orientable $5$-manifold. By Theorem \ref{GR_Odd_Dim_Manifolds}, $H_1(X;Z) = \Z$ and $H_i(X;\Z)$ is finite for all $0<i<5$. Combining this finiteness condition with the fact that $X$ is non-orientable, we get $H_4(X;\Z) = \Z_2$. Finally, by Lemma \ref{PreliminaryResults-2}, $H_3(X;\Z)$ cannot have any torsion, and since it must also be finite, so $H_3(X;\Z) = 0$.
		Using the above observations, we deduce that $X$ has homology groups as given in the statement of Theorem \ref{5-Manifold-Non-Orientable}. Conversely, let the homology groups of $X$ be as given in the statement of Theorem \ref{5-Manifold-Non-Orientable}. Then it is immediate that $X$ is $C$-trivial as the cohomology groups of even degree are zero. This completes the proof of the theorem.
	\end{proof}
	\begin{proof}[Proof of Theorem \ref{6-Manifold-NonOrientable}] We first prove the necessary conditions for $X$ to be $C$-trivial. Observe the following:
		\begin{itemize}
			\item[(i)] Since $X$ is non-orientable, $H_6(X;\Z) = 0$ and $H_5(X; \Z) = \Z^{e_5} \oplus \Z_2$ as its torsion subgroup is $\Z_2$.
			\item[(ii)] By Lemma \ref{PreliminaryResults-2}, $H_2(X;\Z)$ and $H_4(X;\Z)$ are finite abelian groups, say $F$ and $F'$ respectively. Additionally, $H_1(X;\Z)$ and $H_3(X;\Z)$ are free abelian groups, say $\Z^{e_1}$ and $\Z^{e_3}$ for some $e_1,e_3 \ge 0$.
		\end{itemize}
		Using the above observations we deduce that $X$ must have the following homology
		\newline
		$$
		H_i(X;\Z) = \begin{cases}
			\Z  & \text{if } i = 0\\
			\Z^{e_1} & \text{if } i = 1\\
			F & \text{if } i = 2\\
			\Z^{e_3}& \text{if } i = 3\\
			F' & \text{if } i = 4\\
			\Z^{e_5} \oplus \Z_2 & \text{if } i = 5\\  
			0 & \text{if } i = 6.
		\end{cases}
		$$
		\newline
		Computing the homology with $\Z_2$ coefficients and using Poincar\'e duality we get the additional conditions
		$$d_1 \ge 1 \text{ and } \mathrm{Ext}(F,\Z_2) = \mathrm{Ext}(F',\Z_2),$$
		as required. We now prove the converse. Though the above conditions seem much more lax than the previous theorems, it turns out that these are sufficient. To see this, let $X$ be a manifold with the homology as given above. It is clear that $H^2(X;\Z) = H^4(X,\Z) = 0$ and hence for any complex vector bundle $\xi$ of rank less than or equal to two, $c(\xi) = 1$.
		\newline
		Next, let $\xi$ be a complex vector bundle of rank $3$ over $X$ with the Chern classes, $(c_1,c_2,c_3) \in H^2(X;\Z) \times H^4(X;\Z) \times H^6(X;\Z)$. By \cite[Corollary 2.2, Page 276]{Banica_1}, it follows that
		$$c_3 \equiv c_1c_2 + Sq^2c_2 \text{ in } H^6(X;\Z_2).$$
		Since, $c_1 = c_2 = 0$, we have
		$$c_3 \equiv 0 \text{ in } H^6(X;\Z_2).$$ 
		\newline
		The short exact sequence
		$$
		0\to Z\overset{\times 2}{\to} \Z \to \Z_2\to 0
		$$
		gives rise to the long exact sequence in cohomology
		$$\hdots \rightarrow H^5(X;\Z_2) \rightarrow H^6(X;\Z) \rightarrow H^6(X;\Z) \xrightarrow{\rho_2} H^6(X;\Z_2) \rightarrow H^7(X;\Z) \rightarrow \hdots.$$
		Since $H^7(X;\Z) = 0$, $$\rho_2 : \Z_2 \simeq H^6(X;\Z) \rightarrow H^6(X;\Z_2) \simeq \Z_2$$ is an isomorphism. Observe that $c_3 \equiv 0$ in $H^6(X;\Z_2)$ implies that $\rho_2(c_3) = 0$ in $H^6(X;\Z_2)$ and since $\rho_2$ is an isomorphism $c_3 = 0$ in $H^6(X;\Z)$. So, the only option for the Chern classes is $(c_1,c_2,c_3) = (0,0,0)$ and hence $c(\xi) = 1$
		\newline
		Finally, if $\xi$ is a complex vector bundle of rank $k$ greater than or equal to  $4$ over $X$. Then by \cite[Proposition 1.1, Chapter 9]{hus}, we must have $\alpha = \eta \oplus \epsilon$ for some complex vector bundle $\eta$ of rank $3$ and $\epsilon$ a trivial complex bundle. But for $i=1,2,3$, $c_i(\alpha) = c_i(\eta) = 0$ as $\eta$ is a vector bundle of dimension $3$ over $X$. Hence $c(\alpha)=1$. This completes the proof.
	\end{proof}
	We now prove a result which will be used in the proof of Theorem \ref{7_manifold_orientable}.
	\begin{Lem} \label{Lemma-7-Dimension-Orientable}
		Let $X$ be an orientable $7$-manifold. If $X$ is $C$-trivial, then $H^4(X;\mathbb Z)=0$ and $H_6(X;\Z)=0$. 
	\end{Lem}
	\begin{proof}
		As $X$ is orientable, we have that $H^7(X;\mathbb Z)\cong \mathbb Z$ has no $2$-torsion. Also, 
		the homomorphism 
		$$\rho_2:H^6(X;\mathbb Z)\longrightarrow H^6(X;\mathbb Z_2)$$ 
		is an epimorphism. We may now argue as in the proof of Lemma \ref{PreliminaryResults-2}(2) to conclude that 
		$H^4(X;\mathbb Z)=0$. Since $X$ is of odd dimension, it follows from Theorem \ref{GR_Odd_Dim_Manifolds} that $H_6(X;\Z)$ must be a finite group. But orientability of $X$ implies that $H_6(X;\Z)$ cannot have any torsion as well. Hence, $H_6(X;\Z) = 0$. This completes the proof.
	\end{proof}
	
	\begin{proof}[Proof of Theorem \ref{7_manifold_orientable}] Let $X$ be a $7$-dimension orientable $C$-trivial manifold. We make the following observations
		\begin{itemize}
			\item[(i)] By Theorem \ref{GR_Odd_Dim_Manifolds}, $H_7(X;\Z) = \Z$ and $H_i(X;\Z)$ is finite for $1 \le i \le 6$.
			\item[(ii)] By Lemma \ref{PreliminaryResults-1} (1), $H^2(X;\Z) = 0$. Hence, $H_2(X;\Z) = F$ and $H_1(X;\Z) = \Z^{e_1}$ for some $e_1 \ge 0$. But by Theorem \ref{GR_Odd_Dim_Manifolds}, $H_1(X;\Z)$ is a finite group and hence $H_1(X;\Z)= 0$.
			\item[(iii)] By Lemma \ref{Lemma-7-Dimension-Orientable}, $H^4(X;\Z) = 0$, and as a result, $H_4(X;\Z) = F'$ and $H_3(X;\Z) = \Z^{e_3}$ for some $e_3 \ge 0$. Once again, by Theorem \ref{GR_Odd_Dim_Manifolds} $H_3(X;\Z)$ is a finite group which means $H_3(X;\Z)= 0$.
		\end{itemize}
		Using the above observations, the homology groups are as follows
		$$
		H_i(X;\Z) = \begin{cases}
			\Z  & \text{if } i = 0\\
			0 & \text{if } i = 1\\
			F & \text{if } i = 2\\
			0 & \text{if } i = 3\\
			F' & \text{if } i = 4\\
			F'' & \text{if } i = 5\\  
			0 & \text{if } i = 6\\
			\Z & \text{if } i = 7.
		\end{cases}
		$$
		By Poincar\'e duality applied on integral homology and cohomology, we further deduce that $F'' = 0$ and $F'=F$, as required. The converse is immediate. This completes the proof.
	\end{proof}
	\begin{Exm} \label{Orientable_5_7_Manifold_Example}
		Examples of manifolds having homology groups as in Theorems  \ref{5-Manifold-Orientable} and \ref{7_manifold_orientable} are provided in \cite[Example 7, page 232]{ruberman}. There, the author constructs manifolds in all dimension greater than $5$, in which the only homology (apart from dimensions $0$  and $n$) is $\Z_k$ in dimensions $2$ and $n-3$. In particular, the author provides a simply connected $7$-manifold which has integral homology groups as $\Z$ in degrees $0$ and $7$ and $\Z_k$ in degree $4$ and $2$. This has the form given in Theorem \ref{7_manifold_orientable} and hence is $C$-trivial. Additionally, the author constructs a simply connected $5$-manifold which has integral homology groups as $\Z$ in degrees $0$ and $5$ and $\Z_k \oplus \Z_k$ in degree $2$. This has the form given in Theorem \ref{5-Manifold-Orientable} and hence is $C$-trivial.
	\end{Exm}
	
	\begin{Lem} \label{Lemma-7-Dimension-Non-Orientable}
		Let $X$ be a finite CW-complex of dimension $n>3$ with $H^2(X;\Z) = 0$ and let $\xi$ be a complex vector bundle  over $X$. Then $\xi$ restricted to $X^3$ is trivial, where $X^3$ is the $3$-skeleton of $X$. 
	\end{Lem}
	\begin{proof} Note that if $H^2(X;\Z) = 0$, then $H^2(X^3;\Z)=0$. Suppose that $\rk(\xi) = 1$. As first Chern class establishes an isomorphism between isomorphism classes of line bundles on $X$ and $H^2(X;\Z)$, we immediately deduce that $\xi$ must be trivial. Now, consider $\rk(\xi) = k>1$. The restriction $\xi|_{X_3}$ satisfies $3 \le 2k-1$ for all $k > 1$. It follows by \cite[Proposition 1.1, Pg 111]{hus} that $\xi|_{X^3} \simeq \eta \oplus \epsilon^1$ for some complex vector bundle $\eta$ of rank $k-1$ over $X^3$. If $k = 2$, then we stop, else applying the same result again we get $\xi|_{X^3} \simeq \eta' \oplus \epsilon^2$, where $\eta'$ is a vector bundle of rank $k-2$ over $X^3$. The process stops when we can write $\xi|_{X^3} \simeq \eta \oplus \epsilon^{k-1}$ with $\eta$ a complex line bundle over $X^3$. Since $\eta$ is a line bundle over $X^3$ it must be trivial. This completes the proof.
	\end{proof}
	\begin{proof}[Proof of Theorem \ref{7_manifold_non_orientable}] We first show that the homology groups of a $C$-trivial non-orientable $7$-manifold must satisfy the criteria given in the theorem. By Theorem \ref{GR_Odd_Dim_Manifolds}, we already know that 
		\begin{align*}
			H_1(X;\Z) = \Z,\; H_6(X;\Z) = \Z_2, \;H_7(X;\Z) = 0 
		\end{align*}
		and all other homology groups are finite. Consider the third differential of the Atiyah-Hirzebruch spectral sequence
		\begin{align*}
			d_3:H^4(X;\Z) \rightarrow H^7(X;\Z) = \Z_2.
		\end{align*}
		Note that if $\alpha \in \Ker(d_3)$, then $\alpha \in \Ker(d_{2k+1})$ for all $k \ge 1$. Additionally, by Theorem \ref{BoundaryMap_AHSS_Kernel}, if $\alpha \in \Ker(d_{2k+1})$ for all $k \ge 1$, then there exists $\xi$ such that $c_2(\xi) = \alpha$. It follows by $C$-triviality that $d_3$ must be injective. Consequently, 
		\begin{align*}
			H^4(X;\Z) = \Z_2^r,\;r=0,1
		\end{align*}
		and hence $H_3(X;\Z) = \Z_2^r$. Consider the odd differential $d_{2k+1}$ starting at $H^6(X;\Z)$ for $k\ge 1$. As this differential is 0, $\Ker{d_{2k+1}} = H^6(X;\Z)$ for all $k\ge 1$.
		It follows from Theorem \ref{BoundaryMap_AHSS_Kernel} that for all $x \in H^6(X;\Z)$, we must have $2 x = 0$. It then follows that $H_5(X;\Z) = \Z_2^{t_5}$ for some $t_5 \ge 0$.
		\newline
		Combining all these observations we have the tentative integral and mod 2 homology as follows
		$$
		H_i(X;\Z) = \begin{cases}
			\Z  & \text{if } i = 0\\
			\Z & \text{if } i = 1\\
			F & \text{if } i = 2\\
			\Z_2^{r}, \text{ } r=0,1 & \text{if } i = 3\\
			F' & \text{if } i = 4\\
			\Z_2^{t_5} & \text{if } i = 5\\  
			\Z_2 & \text{if } i = 6\\
			0 & \text{if } i = 7
		\end{cases}
		\; \&\; H^i(X;\Z_2) = \begin{cases}
			\Z_2  & \text{if } i = 0\\
			\Z_2 & \text{if } i = 1\\
			\Z_2^{s} & \text{if } i = 2\\
			\Z_2^{r} \oplus \Z_2^{s}, \text{ } r = 0, 1 & \text{if } i = 3\\
			\Z_2^{s'} \oplus \Z_2^{r} \text{ } r = 0, 1 & \text{if } i = 4\\
			\Z_2^{t_5} \oplus \Z_2^{s'} & \text{if } i = 5\\  
			\Z_2 \oplus \Z_2^{t_5} & \text{if } i = 6\\
			\Z_2 & \text{if } i = 7.
		\end{cases}
		$$
		% \newline
		% The cohomology with $\Z_2$ coefficients is as follows
		% \newline
		% $$
		% H^i(X;\Z_2) = \begin{cases}
			%   \Z_2  & \text{if } i = 0\\
			%   \Z_2 & \text{if } i = 1\\
			%   \Z_2^{s} & \text{if } i = 2\\
			%   \Z_2^{r} \oplus \Z_2^{s}, \text{ } r = 0, 1 & \text{if } i = 3\\
			%   \Z_2^{s'} \oplus \Z_2^{r} \text{ } r = 0, 1 & \text{if } i = 4\\
			%   \Z_2^{t_5} \oplus \Z_2^{s'} & \text{if } i = 5\\  
			%   \Z_2 \oplus \Z_2^{t_5} & \text{if } i = 6\\
			%   \Z_2 & \text{if } i = 7.
			% \end{cases}
		% $$
		Applying mod 2 Poincar\'e duality, we deduce that $t_5=0$ and $s = s'$. This shows that the integral homology has the form
		$$
		H_i(X;\Z) = \begin{cases}
			\Z  & \text{if } i = 0\\
			\Z & \text{if } i = 1\\
			F & \text{if } i = 2\\
			\Z_2^r & \text{if } i = 3\\
			F' & \text{if } i = 4\\
			0 & \text{if } i = 5\\  
			\Z_2 & \text{if } i = 6\\
			0 & \text{if } i = 7
		\end{cases}
		$$
		with $r=0,1$.
		If $r=1$ and $X$ is $C$-trivial, then we know that the map 
		$$d_3 : \Z_2 = H^4(X;\Z) \rightarrow H^7(X;\Z) = \Z_2$$ must be injective and hence an isomorphism. Recall that $d_3$ can be given by the following composition $$ H^4(X;\Z) \xrightarrow[]{\rho_2} H^4(X;\Z_2) \xrightarrow[]{\mathrm{Sq^2}} H^6(X;\Z_2) \xrightarrow[]{\tilde{\beta}} H^7(X;\Z)$$
		where $\tilde{\beta}$ is the connecting homomorphism of the long exact sequence induced by 
		$$0 \rightarrow \Z \overset{\times 2}{\rightarrow} \Z \rightarrow \Z_2 \rightarrow 0.$$
		The long exact sequence tells us that that $\tilde{\beta}:H^6(X;\Z_2) \rightarrow H^7(X;\Z)$ is an isomorphism. As $H^4(X;\Z) = H^6(X;\Z_2) = \Z_2$, it follows that $\mathrm{Sq}^2\circ\rho_2$ must be injective and hence the identity map. This proves the forward direction.
		\newline
		Conversely, if $X$ has the given homology, then $H^2(X;\Z) = H^6(X;\Z) = 0$ is immediate. So the only non-zero Chern classes that can exist must be in $H^4(X;\Z) = \Z_2^r$. We now complete the proof in each of the two cases on $r$. If $(i)$ is true then $r=0$ and there are no Chern classes in $H^4(X;\Z)$ as well, making $X$ a $C$-trivial manifold. On the other hand, if $(ii)$ is true and hence $r=1$ together with the composition $\mathrm{Sq}^2 \circ \rho_2$ being the identity map, then $d_3 : H^4(X;\Z) \rightarrow H^7(X;\Z)$ is the identity map. Now, let $\xi$ be any complex vector bundle over $X$, such that $c_2(\xi) = \alpha \in H^4(X;\Z)$ and since $c_i(\xi) = 0$ for all $i \neq 2$, we get that $\ch(\xi) = \alpha$. Observe that $\mathrm{dim}(X) = 7 \ge 3$ and $H^2(X;\Z) = 0$, so we can apply Lemma \ref{Lemma-7-Dimension-Non-Orientable} to conclude that $\xi|_{X^3}$ is trivial. Now, we can apply Theorem \ref{ChernClassExistence_Griffiths} to deduce that $d_3(\alpha) = 0$. But since $d_3$ is identity, hence $\alpha = 0$, showing that $X$ is $C$-trivial. This completes the proof.
	\end{proof}
	
	\section{Applications of the results}
	In this section we give some applications of the homological classification obtained for low dimensional $C$-trivial manifolds.
	\begin{Thm}
		Let $X$ be an orientable, $C$-trivial odd dimensional manifold. There are no manifolds $Y$ and $Z$ of positive dimension for which $X = Y \times Z$
	\end{Thm}
	\begin{proof}
		Since $X$ is odd dimensional and orientable then at least one among $Y$ and $Z$ must be even dimensional and orientable, hence not $C$-trivial. The result then follows.
	\end{proof}
	
	Recall that $M\# N$ denotes the connected sum of $M$ and $N$. In the following result, we are concerned with the problem of relating $C$-triviality of $M\#N$ to that of $M$ and $N$.
	\begin{Thm} Let $X = M \# N$ be the connected sum of two compact $n$-manifolds.
		\begin{itemize}
			\item[(1)] If $n$ is odd, and $M, N$ both are non-orientable then $X$ cannot be $C$-trivial.
			\item[(2)] If $n=3,5,7$ and $X$ is orientable, then $X$ is $C$-trivial iff $M$ and $N$ are $C$-trivial.
			\item[(3)] If $n=3,5$ and $X$ is non-orientable, then $X$ is $C$-trivial iff $M,N$ are both $C$-trivial and exactly one of $M$ or $N$ is orientable and the other is non-orientable.
		\end{itemize}
	\end{Thm}
	\begin{proof} We first begin by proving (1). If possible, let $X$ be $C$-trivial. For an odd dimensional non-orientable manifold $H_{n-1}(X;\Z) = \Z_2 \oplus \Z^e$ and by Theorem \ref{GR_Odd_Dim_Manifolds} $H_{n-1}(X;\Z)$ is finite and hence, $H_{n-1}(X;\Z) = \Z_2$. However, since $M$ and $N$ are both non-orientable $H_{n-1}(M \# N;\Z)$ must contain $\Z$ as a subgroup. This contradiction proves (1).
		\newline
		For (2), the proofs for $n=3,5$ and $7$ are similar, so we give the proof of $n = 7$ as an example. Since $X = M \# N$ is orientable, both $M$ and $N$ must be orientable and $H_i(M \# N, \Z) = H_i(M ; \Z) \oplus H_i(N, \Z)$ for all $i \neq 0,7$. Assume first that $X$ is $C$-trivial. By the homological classification of $7$-dimensional orientable $C$-trivial manifolds given in Theorem \ref{7_manifold_orientable}, we have that $H_i(M;\Z) = H_i(N;\Z) = 0$ for $i = 1,3,5,6$ and $H_i(M;\Z)$ and $H_i(N;\Z)$ are finite abelian groups such that $H_i(M;\Z) \oplus H_i(N;\Z) = F$ for $i = 2,4$. Since $M$ and $N$ are orientable, using Poincar\'e duality we get that $H_2(M;\Z) = H_4(M;\Z)$ and similarly for $N$. This means that $M$ and $N$ are $7$-dimensional manifolds that have the form given in Theorem \ref{7_manifold_orientable} and therefore are $C$-trivial. The converse is immediate. This proves statement (2) for $n=7$.
		\newline
		Next we prove (3) for the $n=5$ case. Let $X$ be a $C$-trivial non-orientable $5$-manifold. Since $X$ is non-orientable at least one among $M$ and $N$ must be non-orientable. By (1), both cannot be non-orientable and so exactly one among $M$ and $N$ is non-orientable. Without loss of generality let $M$ be non-orientable. By using the homological classification of $5$-dimensional of $5$-dimensional non-orientable $C$-trivial manifolds given in Theorem \ref{5-Manifold-Non-Orientable}, we observe the following:
		\begin{itemize}
			\item[(i)] $H_3(M;\Z) \oplus H_3(N;\Z) = H_3(X;\Z) = 0$ and hence $H_3(M;\Z) = H_3(N;\Z) = 0$.
			\item[(ii)] $H_4(M;\Z) \oplus H_4(N;\Z) = H_4(X;\Z) = \Z_2$ and using the fact that $M$ is orientable, we must have $H_4(M;\Z) = 0$ and $H_4(N;\Z) = \Z_2$.
			\item[(iii)] Using (i) and (ii), $H^4(M;\Z) = 0$ and then using Poincar\'e duality, we must have $H_1(M;\Z) = 0$. Since $H_1(M;\Z) \oplus H_1(N;\Z) = H_1(X;\Z) = \Z$, $H_1(N;\Z) = \Z$.
		\end{itemize}
		Combining the above observations and using the Theorems \ref{5-Manifold-Orientable} and \ref{5-Manifold-Non-Orientable} it is clear that $M$ and $N$ are $C$-trivial. The converse follows by similar arguments as above, completing the proof of the theorem.
	\end{proof}
	\begin{Rem} Note that for $n=6$, it is possible that $M\# N$ is $C$-trivial though neither $M$ nor $N$ is $C$-trivial. For example, if $M$ is the $6$-sphere and $N$ is any $6$-dimensional non-orientable $C$-trivial manifold, then $M \# N$ is $C$-trivial as it has the same homology as $N$ even though $M$ is not $C$-trivial.
	\end{Rem}

	\section*{Acknowledgment}
	The authors thank Dr. Aniruddha Naolekar at the Indian Statistical Institute, Bangalore, for suggesting the problem and for many fruitful discussions.
	
	\section*{Declarations}
	
	\subsection*{Data availability} We do not analyse or generate any datasets, because our work proceeds
	within a theoretical and mathematical approach.
	
	\subsection*{Conflict of interest} On behalf of both authors, the corresponding author states that there is
	no conflict of interest.

\end{document}